\documentclass{amsproc}

\usepackage{amssymb}
\usepackage{amsfonts}
\usepackage{amsmath}
\usepackage[latin1]{inputenc}
\usepackage{color}
\usepackage[mathscr]{euscript}
\usepackage{enumerate}

\numberwithin{equation}{section}

\newtheorem{theorem}[equation]{Theorem}
\newtheorem{lemma}[equation]{Lemma}
\newtheorem{proposition}[equation]{Proposition}

\theoremstyle{definition}
\newtheorem{definition}[equation]{Definition}

\theoremstyle{remark}
\newtheorem{remark}[equation]{Remark}

\newcommand{\lb}[1]{\text{$\mathscr{#1}$}}
\newcommand{\dgraphg}[1]{\text{$\lb{#1}$}}									
\newcommand{\dgraphupleg}[3]{\text{$\dgraphg{#1}=(\dgraphg{#1}^0,\dgraphg{#1}^1, #2, #3)$}}	

\newcommand{\dgraph}{\text{$\dgraphg{E}$}}									
\newcommand{\dgraphuple}{\text{$\dgraphupleg{E}{r}{d}$}}					
\newcommand{\scj}{\subseteq}
\newcommand{\nn}{\mathbb{N}}
\newcommand{\til}[1]{\widetilde{#1}}
\newcommand{\cidlg}[1]{#1\mathrm{-Idl}}
\newcommand{\cidl}{\cidlg{C}}

\begin{document}
	
\title[The boundary path space of a topological graph]{Recovering the boundary path space of a topological graph using pointless topology}
\author{Gilles G. de Castro}
\address{Departamento de Matemática, Centro de Ciências Físicas e Matemáticas, Universidade Federal de Santa Catarina, 88040-970 Florianópolis SC, Brazil.}
\email{gilles.castro@ufsc.br}

\begin{abstract}
	First, we generalize the definition of a locally compact topology given by Paterson and Welch for a sequence of locally compact spaces to the case where the underlying spaces are $T_1$ and sober. We then consider a certain semilattice of basic open sets for this topology on the space of all paths on a graph and impose relations motivated by the definitions of graph C*-algebra in order to recover the boundary path space of a graph. This is done using techniques of pointless topology. Finally, we generalize the results to the case of topological graphs.
\end{abstract}

\maketitle

%
%

\section{Introduction}

Topological graphs and their C*-algebras were introduced by Katsura as a way of generalizing both graph C*-algebras and homeomorphism C*-algebras \cite{MR2067120}. One usually would like to describe a C*-algebra as a groupoid C*-algebra in order to use techniques introduced by Renault \cite{MR584266}. For topological graph C*-algebras, this was done by Yeend \cite{MR2301938}. The unit space of Yeend's groupoid is the boundary path space of the topological graph (see \cite{MR3679613} for the proof that Yeend's definition and the one we use here are the same).

To understand the necessity of the boundary path space, we start by considering finite graphs. In this case, one can associate a compact space of infinite paths as a subspace of a a infinite product of finite sets, namely the vertices set. These are examples of one-sided shift spaces, and more specifically shifts of finite type \cite{MR1369092}. The class of C*-algebras associated to shifts of finite type are called Cuntz-Krieger algebras \cite{MR561974}. One of the results of \cite{MR1432596} by Kumjian et. al. was to generalize Cuntz-Krieger algebras to infinite graphs. Their approach used groupoids and for that they needed a generalization of the shift space of a graph. They accomplished that by not allowing the graph to have sources of infinite receivers. In this case, even if the infinite product of the set of vertices is not locally compact, the subspace of paths is.

The C*-algebra of a graph was later generalized to arbitrary graphs by Fowler, Laca and Raeburn using generators and relations to define the C*-algebra \cite{MR1670363}. The groupoid description of these C*-algebras in the case of graphs with no sources was done by Paterson using inverse semigroups \cite{MR1962477}. The key idea was that we need not only to consider infinite paths, but also some finite paths. With a certain topology on the set of all paths, infinite paths could converge to finite paths. This happens exactly in the presence of infinite receivers. We could later allow the graph to have sources by adding some isolated points in the path spaces. The end result is what is now known as the boundary path space \cite{MR3119197}.

To find the boundary path space in \cite{MR1962477}, Paterson first considered the set of all paths and reduced to the boundary path space by looking at the structure of the graph. One of Exel's main goal in \cite{MR2419901} was to obtain this reduction looking solely at the graph inverse semigroup without the knowledge of the graph. Exel then introduced the notion of the tight spectrum of an inverse semigroup. In the particular case of graph inverse semigroup, the tight spectrum is essentially the boundary path space of the graph. The author together with Boava and Mortari generalized this result to the more general context of labelled graphs \cite{MR3648984}.

One question that arises is if we can use Exel's framework in the case of topological graphs. As we mentioned in the first paragraph, we do have a boundary path space for a topological graph. The main problem is that the tight spectrum is totally disconnected while the boundary path space may not be. Central to this problem is the fact that the framework developed by Exel in \cite{MR2419901} is closely related to Boolean algebras, which in turn, by Stone duality, gives a totally disconnected space. However, one should be able to study the boundary path space of a topological graph using inverse semigroups due to the works of Resende \cite{MR2304314}, and later of Lawson and Lenz \cite{MR3077869}. There is a duality between a certain class of groupoids, of which Yeend's groupoid are part, and a certain class of inverse semigroups called pseudogroups. This duality is closed related to the duality seen in pointless topology between a certain class frames and a certain class of topological space as mentioned by Lawson and Lenz in \cite{MR3077869}.

The interesting thing about frames is that they can be presented by generators and relations \cite{MR698074}. The main goal of this paper is to define relations on the frame of opens sets of the set of all paths on a graph using as motivation one of the relations used to define the graph C*-algebra. We do this in such a way that it can be generalized to topological graphs.

The structure of the paper is as follows: in Section \ref{section:frame}, we recall some definitions and results from pointless topology used throughout the paper; in Section \ref{section:topology} we study the Paterson-Welch topology \cite{MR2146225} as being the patch topology of the topology given by cylinder sets (the former being the topology used on the space of all paths on a graph); in Section \ref{section:graph} we show how to obtain the boundary path of a discreet graph using pointless topology, which is then generalized to topological graphs in Section \ref{section:topolgical.graph}.

Some remarks on notation: we consider $0\in\nn$ and write $\nn^*$ for $\nn\setminus\{0\}$; $X^c$ represents the complement of a set $X$; and $[\cdot]$ represents the Iverson's brackets that returns $0$ if the argument is false and returns $1$ if the argument is true.

%
%

\section{Frames generated by semilattices and relations}\label{section:frame}

We start by recalling some terminology used in pointless topology \cite{MR698074,MR2868166,MR1002193}. Given a partially ordered set $(P,\leq)$, we say that $P$ is a (meet-)semilattice if the infimum of any two elements $a,b\in P$ exist, which is then called the meet of $a$ and $b$, and denoted by $a\wedge b$. We say that $P$ is a complete lattice, if the infimum and supremum of an arbitrary subset $I\scj P$ exist; they are called the meet and join of $I$ and denoted by $\bigwedge I$ and $\bigvee I$ respectively. A frame $F$ is complete lattice satisfying the following distributivity condition
\[a\wedge\bigvee I=\bigvee \{a\wedge b:b\in I\}\]
for all $a\in F$ and all $I\scj F$. A semilattice homomorphism is a function between semilattices that preserves finite meets, and a frame homomorphism is a function between frames that preserve finite meets and arbitrary joins.

The frame $\{0<1\}$ is denoted by $\mathbf{2}$ and it can be seen as the unique topology of a one-point space. A point in $F$ is a frame homomorphism $f:F\to\mathbf{2}$ and the set of all point in $F$ is denoted by $\mathrm{pt}(F)$. For each $a\in F$, define $\Sigma_a=\{f\in\mathrm{pt}(F):f(a)=1\}$, then $\{\Sigma_a\}_{a\in F}$ is a topology on $F$, and $\mathrm{pt}(F)$ with this topology is called the spectrum $F$. For a topological space $X$, its topology $\Omega(X)$ is a frame, and we say that $X$ is sober if $X$ is homeomorphic to $\mathrm{pt}(\Omega(X))$ (it is actually sufficient to show that the map $x\in X\to [x\in U]\in\mathrm{pt}(\Omega(X))$ is a bijection). On the other hand, a frame $F$ is called spatial if is isomorphic to the topology of a topological space and in this case, $F$ is isomorphic to $\Omega(\mathrm{pt}(F))$. Both $\Omega$ and $\mathrm{pt}$ can be made into functors and we arrive at an equivalence between the category of sober spaces with continuous functions and spatial frames with frame homomorphisms.

For the purposes of this paper, we are interested in defining a frame from a semilattice and some ``join relations''. This is done with the notion of a site given by Johnstone \cite{MR698074}. For a semilattice $A$, a coverage $C$ assigns to each $a\in A$ a set $C(a)$ of subsets of the lower set of $a$, $\downarrow a=\{b\in A:b\leq a\}$, called covering of $a$, in such a way that $S\in C(a)$ implies that $\{s\wedge b:s\in S\}\in C(b)$ for all $b\leq a$ (meet-stability condition). The pair $(A,C)$ is called a site.

From a site, one can define a frame freely generated by $(A,C)$ as a frame $B$ together with semilattice homomorphism $f:A\to B$ satisfying
\begin{equation}\label{eq:site}
f(a)=\bigvee_{s\in S}f(s)
\end{equation}
for all $a\in A $ and $S\in C(a)$, which is universal in the sense that if $B'$ is another frame with a semilattice homomorphism $f':A\to B'$ satisfying (\ref{eq:site}), then there is a unique frame homomorphism $g:B\to B'$ such that $f'=g\circ f$. One way to think is that we are imposing the relation $a=\bigvee S$.

To show the existence of such a frame, we use the notion of a $C$-ideal, which are subsets $I\scj A$ that are lower closed and such that if $I$ contains a covering of an element $a\in A$, then $a\in I$. The set of all C-ideals is denoted by $\cidl$ and the frame structure on $\cidl$ comes from the order by given by inclusion. In this case, the meet is given by the intersection and the join is the C-ideal generated by the union, that is, the intersection of all C-ideals containing the union \cite[Theorem~4.4.2]{MR1002193}. The map $f:A\to\cidl$ is such that $f(a)$ is the $C$-ideal generated by $a$, for any $a\in A$. Observe that if $I\in\cidl$, then $\bigvee_{a\in I}=I$.

The following lemma links the above idea with another way of thinking a frame as being generated by a semilattice.

\begin{lemma}\label{lemma:semilattice.generates.frame}
	Let $F$ be a frame and suppose that $A\scj F$ is a semilattice that generates $F$ in the sense that for all $x\in F$ there exists $B\scj A$ such that $x=\bigvee B$. Then defining $C(a)$ to be the set of all subsets $B\scj \downarrow a \cap A$ such that $\bigvee B=a$, then $C$ is a coverage on $A$ such that $F\cong \cidl$ as frames.
\end{lemma}

\begin{proof}
	That $C$ is a coverage follows immediate from the distributivity of meets over arbitrary joins.
	
	Now, by the universal property of $\cidl$, we find the map $g:\cidl\to F$ given by $g(I)=\bigvee I$, which is a frame homomorphism that is surjective because $A$ generates $F$.
	
	Let $I,J\in \cidl$ be such that $\bigvee I=\bigvee J$, then by \cite[Proposition~4.4.2.2]{MR1002193}
	\[\bigvee(I\cap J)=\bigvee(I\wedge J)=\bigvee I\wedge\bigvee J=\bigvee I.\]
	
	Suppose now that $I\neq J$ and, without loss of generality, that $J\scj I$. Since $J$ is a lower set in $A$, there exists $a\in I$ such that $j<a$ for all $j\in J$. This implies that $\bigvee J\leq a\leq \bigvee I$. If $\bigvee J = a$, then $J$ is a cover of $a$, but since $J$ is a $C$-ideal, we would have that $a\in J$, which is a contradiction. If $\bigvee J < a$, then $\bigvee J < \bigvee I$, which is also a contradiction. If follows that $g$ is injective.
\end{proof}

In the context of pointless topology, the notion of subspace is substituted by sublocale. One way of defining a sublocale is by considering maps $\nu:F\to F$ on a frame $F$ satisfying $\nu(a\wedge b)=\nu(a)\wedge\nu(b)$, $a\leq \nu(a)$ and $\nu(\nu(a))\leq\nu(a)$ for all $a,b\in F$, called nuclei. The image of a nucleus $\nu$, is proven to be the set of fixed points of $\nu$ and it is a frame \cite[II.2.2]{MR698074}; also, it is called a sublocale of $F$.

One way of finding sublocales is by imposing new relations: suppose that $F$ is a frame generated by $A$ as in Lemma \ref{lemma:semilattice.generates.frame} and that $C'$ is another coverage on $A$. We can define a new coverage $\overline{C}=C\cup C'$ such that for each $a\in A$, $\overline{C}(a)=C(a)\cup C'(a)$. To impose the new relations on $F$, we consider the map $\nu:\cidl\to\cidl$ given by
\begin{equation}\label{eq:nucleus}
	\nu(I)=\bigcap_{\substack{I\scj I',\,I'\in\cidlg{\overline{C}}}}I'
\end{equation}
which can be proven to be a nucleus similar to what is done in \cite[II.2.11]{MR698074}. Since every $\overline{C}$-ideal is a $C$-ideal, the image of $\nu$ is exactly $\cidlg{\overline{C}}$. Using Lemma \ref{lemma:semilattice.generates.frame}, we can think that $\nu$ is defined on $F$.

In the case that $F$ is the topology $\Omega(X)$ on a set $X$, if $\mathcal{B}\scj \Omega(X)$ is a semilattice, then it is a basis for this topology. We can use the discussion above to impose relations on the elements of this basis and find a sublocale $G\scj \Omega(X)$, which is the image of the nucleus $\nu$ given by (\ref{eq:nucleus}). Restricting the co-domain of $\nu$ to $G$, we get a frame homomorphism surjection $f:\Omega(X)\to G$ \cite[II.2.2]{MR698074}. In the sober case, since there is a homeomorphism $\Phi:\mathrm{pt}(\Omega(X))\to X$, we can define a subspace $Y=\{\Phi(p\circ f)\,|\,p\in\mathrm{pt}(G)\}$. If $G$ is a spatial frame, then it is isomorphic to the induced topology on $Y$ \cite[Chapter~VI]{MR2868166}.

\begin{lemma}\label{lemma:subspace.generated.by.relations}
	In the conditions above, suppose also that $X$ is $T_1$. Let $C$ be the coverage given by $\mathcal{B}$ as in Lemma \ref{lemma:semilattice.generates.frame}, $C'$ the coverage defining the new relations and $\overline{C}=C\cup C'$. Then $Y=\{x\in X\,|\,\{x\}^c\in\cidlg{\overline{C}}\}$.
\end{lemma}

\begin{proof}
	We use the fact that in a $T_1$ space, singletons are closed sets.
	
	First suppose that $x=\Phi(p\circ f)$ for some $p\in\mathrm{pt}(G)$. For $U=\{x\}^c$, since $U\scj f(U)$, either $f(U)=U$ or $f(U)=X$. If $f(U)=X$, then $p(f(U))=p(X)=1$, because $X=1_G$ and $p$ is a frame homomorphism. On the other hand $p(f(U))=\Phi^{-1}(x)(U)=[x\in U]=0$. This is a contradiction so that $f(U)=U$, which means that $\{x\}^c=U\in \cidlg{\overline{C}}$.

	Now suppose that $\{x\}^c\in \cidlg{\overline{C}}$ so that $f(\{x\}^c)=\{x\}^c$ and define $p\in\mathrm{pt}(G)$ by $p(U)=[x \in U]$ for $U\in G$. We have to prove that $p(f(U))=[x \in U]$ for all $U\in\Omega(X)$. We know that $U\scj f(U)$. If $x\in U$, then $p(f(U))=[x\in f(U)]=1=[x\in U]$. And if $x\notin U$, then $U=U\cap \{x\}^c$ so that $f(U)=f(U\cap \{x\}^c)=f(U)\cap f(\{x\}^c)=f(U)\cap \{x\}^c$. In this case $p(f(U))=[x\in f(U)]=0=[x\in U]$.
\end{proof}

\begin{remark}\label{remark:subspace}
	Due to Lemma \ref{lemma:semilattice.generates.frame}, observe that the elements of $Y$ given in Lemma \ref{lemma:subspace.generated.by.relations} are exactly those $x\in X$ such that for any $U\in\mathcal{B}$ and any $S\in C'(U)$, $x\notin V$ for all $V\in S$ implies that $x\notin U$, or equivalently, $x\in U$ implies that $x\in V$ for some $V\in S$.
\end{remark}

%
%

\section{The Paterson-Welch topology}\label{section:topology}

Let $\{X_n\}_{n\in\nn^*}$ be a sequence of locally compact $T_1$ sober spaces. In general, the product space $\prod_{n\in\nn^*}X_n$ is not locally compact. By looking at the path spaces of graphs and k-graphs, Paterson and Welch in \cite{MR2146225} defined a topology on the space
\[W=\bigcup_{k\in\nn^*}(X_1\times\cdots\times X_k) \cup \prod_{n\in\nn^*}X_n.\]
They asked the spaces to be locally compact metric spaces and used the Alexandroff extension on each $X_n$ to arrive at a compact space $X_n^{\infty}$. By Tychnoff's theorem $A:=\prod_{n\in\nn^*}X_n^{\infty}$ is compact, and by defining a surjection $Q:A\to W_0$, where $W_0=\{0\}\cup W$ for $0$ an element not belonging to $W$, they gave the quotient topology on $W_0$, which they proved was a compact metric space. By removing $0$, they obtained a locally compact metrizable topology on $W$.

The goal of this section is to arrive at the Paterson-Welch topology on $W$ in a different way and with the weaker hypothesis that the spaces are $T_1$ and sober instead of being metric spaces. Also, since we are interested in the path space of graphs, we will work with a certain subset of $W$. For each $k\in\nn^*$, let $Y_k$ be a closed subspace of $X_1\times\cdots\times X_k$, with the property that if $k\leq l$, then $(x_1,\ldots,x_l)\in Y_l$ implies that $(x_1,\ldots,x_k)\in Y_k$, so that we have well-defined continuous projections $\pi_{k,l}:Y_l\to Y_k$. Notice that each $Y_k$ is also $T_1$, sober and locally compact. We define $Y_{\infty}$ as the set of all sequences $(x_1,x_2,\ldots)\in\prod_{n\in\nn^*}X_n$ such that $(x_1,\ldots,x_k)\in Y_k$ for all $k\in\nn^*$, and finally define
\[Y=\bigcup_{k\in\nn^*}Y_k\cup Y_{\infty}.\]

For $x\in Y$, $\ell(x)$ denotes the length of the sequence $x$, so that $\ell(x)=k$ if $x\in Y_k$ for $k\in \nn^*\cup\{\infty\}$. If $A\scj Y$ is such that $A\scj Y_k$ for some $k\in \nn^*\cup\{\infty\}$, we also write $\ell(A)=k$.

\begin{definition}
	For a set $A\scj Y_k$, $k\in\nn^*$, we define the cylinder set $Z(A)$ as the subset of $Y$ of all finite and infinite sequences such the sequence of the first $k$ coordinates belong to $A$. If $A$ is open, we say that $Z(A)$ is an open cylinder set. Also if $A \scj \bigcup_{k\in\nn^*}Y_k$, we write $Z(A)$ for the union $\bigcup_{k\in\nn^*}Z(A\cap Y_k)$.
\end{definition}

We want to define a topology from the cylinder sets. For that purpose, let $U\scj Y_k$ and $V\scj Y_l$ be open sets and suppose that $k\leq l$. For $A=\pi_{k,l}^{-1}(U)\cap V$, we have that $A$ is open in $Y_l$, and $Z(U)\cap Z(V)=Z(A)$. This means that the cylinder sets of open sets form a basis for a topology on $W$, which we will denote by $\tau_{cyl}$. Notice that the induced topology of $\tau_{cyl}$ on each $Y_k$ coincides with the original topology. Indeed, if $U\scj Y_k$ is open with respect to the original topology, then $U=Z(U)\cap Y_k$ is open on the induced topology; and for $V\scj Y_l$ open, $Z(V)\cap Y_k=\emptyset$ if $l>k$, and $Z(V)\cap Y_k=\pi_{l,k}^{-1}(V)\cap Y_k$ if $l\leq k$, so that $Z(V)\cap Y_k$ is open on the original topology.

Let us call $E$ the semilattice of open cylinder sets. In general, the frame freely generated by $E$ would, in a certain way, lose the topology on each $Y_k$, for example, we would not observe the relation $Z(U)\cup Z(V)=Z(U\cup V)$. For each $U\scj Y_k$, we then consider the covering $\{Z(U_\lambda)\}_{\lambda\in\Lambda}$ of $Z(U)$, where each $U_{\lambda}$ is an open subset of $Y_k$ and $U=\bigcup_{\lambda\in\Lambda}U_{\lambda}$. Using the previous paragraph, it is not difficult to check that these coverings satisfy the meet-stability condition so that we have a coverage $C$ on $E$.

\begin{proposition}
	Let $(E,C)$ be the site defined from the above coverage. Then the frame freely generated by $(E,C)$ is isomorphic to $\tau_{cyl}$.
\end{proposition}

\begin{proof}
	Let $\overline{C}$ be the coverage on $E$ as in Lemma \ref{lemma:semilattice.generates.frame} for $\tau_{cyl}$ with the basis $E$. From this lemma and its proof, it is sufficient to show that $\cidl=\cidlg{\overline{C}}$. Clearly for all $Z(U)\in E$, we have that $C(Z(U))\scj\overline{C}(Z(U))$ and hence $\cidlg{\overline{C}}\scj \cidl$.
	
	Now let $I\in\cidl$ and suppose that $I$ contains a $\overline{C}$-covering $\{Z(V_\lambda)\}_{\lambda\in\Lambda}$ of $Z(U)\in E$, that is $Z(U)=\bigcup_{\lambda\in\Lambda}Z(V_{\lambda})$. Observe that $\ell(V_\lambda)\geq \ell(U)$ for all $\lambda\in\Lambda$ and if $\ell(V_\lambda) > \ell(U)$, then $U\scj Z(U)\cap Z(V_\lambda)^c$. So, if we take $\Lambda'=\{\lambda\in\Lambda\,|\,\ell(V_\lambda)=\ell(U)\}$, we have that $U=\bigcup_{\lambda\in\Lambda'} V_{\lambda}$ and $\{Z(V_\lambda)\}_{\lambda\in\Lambda'}$ is a $C$-covering of $Z(U)$. Since $I\in\cidl$, there exists $\lambda_0\in\Lambda'$ such that $Z(V_{\lambda_0})\in I$, but this implies that $I\in\cidlg{\overline{C}}$.
\end{proof}

The topology $\tau_{cyl}$, in general, is not $T_1$ even if the original spaces are, because we cannot separate a sequence from any of its beginnings. We have the following result, however.

\begin{proposition}
	The space $(Y,\tau_{cyl})$ is sober.
\end{proposition}

\begin{proof}
	We have to prove that the map $\Phi:Y\to \mathrm{pt}(\tau_{cyl})$, given by $\Phi(y)(A)=[y\in A]$ for $A\in\tau_{cyl}$, is a bijection.
	
	We start proving the surjectivity of $\Phi$. Let $p\in\mathrm{pt}(\tau_{cyl})$ and notice that since $p$ preserves order if $p(Z(Y_k))=1$ for some $k\in\nn^*$, then $p(Z(Y_l))=1$ for all $l\leq k$. Suppose first that there exists a maximum $k\in\nn$ such that $p(Z(Y_k))=1$ and $p(Z(Y_{k+1}))=0$. Observe that the map $p_k:\Omega(Y_k)\to \mathbf{2}$ given by $p_k(U)=p(Z(U))$ is a frame homomorphism, and, since $Y_k$ is sober, there exists $y\in Y_k$ such that $p_k(U)=[y\in U]$. We claim that $p(A)=[y\in A]$ for all $A\in\tau_{cyl}$. Notice that because every element in $\tau_{cyl}$ is the union of open cylinder sets, it is sufficient to consider $A$ of the form $Z(V)$ for some $V\scj Y_l$ with $l\in\nn^*$. If $l>k$, then $p(Z(V))=0$ by the definition of $k$ and due to the choice of the family $\{Y_n\}_{n\in\nn^*}$. Also, in this case $y\notin Z(V)$ because every element of $Z(V)$ has length greater or equal to $l$. Now, if $l\leq k$, we have the following:
	\[[y\in Z(V)]= [y\in Z(V)\cap Z(Y_k)]=[ y\in Z(\pi_{k,l}^{-1}(V)\cap Y_k)]= [y\in \pi_{k,l}^{-1}(V)\cap Y_k] =\]
	\[p_k(\pi_{k,l}^{-1}(V)\cap Y_k)= p(Z(\pi_{k,l}^{-1}(V)\cap Y_k))= p(Z(V)\cap Z(Y_k))=\]
	\[p(Z(V))\wedge p(Z(Y_k))= p(Z(V))\wedge 1= p(Z(V)).\]
	
	Now, suppose that $p(Z(Y_k))=1$ for all $k\in\nn^*$. Using the map $p_k$ as above for each $k\in\nn^*$, we find $y_k\in Y_k$ such that $p_k(U)=[y_k\in U]$ for all open subset $U\scj Y_k$. We claim that if $k\leq l$, then $\pi_{k,l}(y_l)=y_k$. Indeed, for an arbitrary open subset $U\scj Y_k$,
	\[[\pi_{k,l}(y_l)\in U]=[y_l \in \pi_{k,l}^{-1}(U)]=p_l(\pi_{k,l}^{-1}(U))=p(Z(\pi_{k,l}^{-1}(U)))=\]
	\[p(Z(U)\cap Z(Y_l))=p(Z(U))\wedge p(Z(Y_l))=p(Z(U))=p_k(U)=[y_k\in U],\]
	and since $Y_k$ is sober, this implies that $\pi_{k,l}(y_l)=y_k$. By the definition of $Y_{\infty}$, there exists $y=(x_1,x_2\ldots)\in Y_{\infty}$ such that $y_k=(x_1,\ldots,x_k)$ for all $k\in\nn^*$. Now, for arbitrary $k\in\nn^*$ and arbitrary $V\scj Y_k$ open
	\[[y\in Z(V)]=[y_k\in V]=p_k(V)=p(Z(V)),\]
	which implies that $p(A)=[y\in A]$ for all $A\in \tau_{cyl}$.
	
	To prove the injectivity of $\Phi$, let $x,y\in Y$ be given and suppose that $x\neq y$. If $\ell(x)<\ell(y)$, then $[x\in Z(Y_{\ell(y)})]=0\neq 1=[y\in Z(Y_{\ell(y)})]$. Analogously if $\ell(x)>\ell(y)$. If $\ell(x)=\ell(y)=k\in\nn^*$, then using that $Y_k$ is sober, we find $U\scj Y_k$ open such that $[x\in Z(U)]=[x\in U]\neq [y\in U]=[y\in Z(U)]$. Finally, if $\ell(x)=\ell(y)=\infty$, since $x\neq y$, there exists $k\in\nn^*$ such that the k-th coordinates are different, and as before, this implies that there exists $U\scj Y_k$ such that $[x\in Z(U)]\neq [y\in Z(U)]$.
\end{proof}

Motivated by by Lawson and Lenz \cite{MR3077869}, we consider the patch topology on $(Y,\tau_{cyl})$. We recall some needed definitions.

\begin{definition}
	Let $X$ be a topological space. We say that a subset $A\scj X$ is saturated if $A$ is the intersection of all open subsets of $X$ containing $A$.
\end{definition}

Observe that if $X$ is $T_1$, since all singletons are closed, every subset of $A\scj X$ is saturated, indeed $A=\bigcap_{x\in A^c}\{x\}^c$.

\begin{definition}
	Let $X$ be a topological space. A cocompact subset of $X$ is the complement of a compact saturated set. The cocompact topology on $X$ is the topology generated by the cocompact subsets of $X$. The patch topology on $X$ is the coarsest topology containing both the original topology and the cocompact topology.
\end{definition}

We will denote the patch topology on $(Y,\tau_{cyl})$ by $\tau_{patch}$. To find a basis for the patch topology $\tau_{patch}$, we need a few lemmas.

\begin{lemma}\label{lemma:compact.saturated.cylinder}
	If $K\scj Y_k$ is compact, then $Z(K)$ is compact and saturated with respect to $\tau_{cyl}$.
\end{lemma}

\begin{proof}
	Let $\{Z(U_\lambda)\}_{\lambda\in \Lambda}$	be a family of open cylinder sets that covers $Z(K)$. Since, $K\scj Z(K)$, the family also covers $K$, but this implies that $\{Z(U_{\lambda})\cap Y_k\}_{\lambda\in \Lambda}$ is an open covering of $K$ in $Y_k$. Since $K$ is compact, there is a finite subcover $\{Z(U_{\lambda_i})\cap Y_k\}_{i=1}^n$ of $K$. Observe that if $Z(U_{\lambda_i})\cap Y_k$ is not empty then $\ell(U_{\lambda_i})\leq k$, which we can suppose it is true for all $i=1,\ldots,n$. It follows that $\{Z(U_{\lambda_i})\}_{i=1}^n$ is a cover of $Z(K)$ and hence $Z(K)$ is compact.
	
	To prove that $Z(K)$ is saturated, we use that $Y_k$ is $T_1$ and hence $K$ is saturated in $Y_k$. If $\mathcal{U}$ is the family of open sets of $Y_k$ containing $K$, then $K=\bigcap_{U\in\mathcal{U}}U$ so that $Z(K)=\bigcap_{U\in\mathcal{U}}Z(U)$, and therefore $Z(K)$ is saturated. 
\end{proof}

\begin{lemma}\label{lemma:compact.saturated.as.union}
	If $\til{K}$ is a compact saturated subset of $Y$ with respect to $\tau_{cyl}$ and $w\in Y\setminus\til{K}$, then there exist compact sets $K_i \scj  Y_{k_i}$ for $i=1,\ldots,n$ and some $k_i\in \nn^*$ such that $\til{K}\scj \bigcup_{i=1}^n Z(K_i)$ and $w\notin\bigcup_{i=1}^n Z(K_i)$.
\end{lemma}

\begin{proof}
	Since $\til{K}$ is saturated, there exists an open set $\til{U}\scj Y$ containing $\til{K}$ and such that $w\notin\til{U}$. We write $\til{U}$ as an union of open cylinder sets and use that $\til{K}$ is compact to find open cylinder sets $Z(U_1),\ldots,Z(U_m)$ such that $\til{K}\scj \bigcup_{i=1}^m Z(U_i)$ and $w\notin Z(U_i)$ for all $i=1,\ldots,m$. Define $k_i=\ell(U_i)$ for each $i=1,\ldots,m$.
	
	Now, since each $Y_k$ is locally compact, for each $i=1,\ldots,m$, and each $x\in Z(U_i)$, there exists a compact neighbourhood $K_{x,i}$ of $(x_1,\ldots,x_{l_i})\in Y_{l_i}$ with $K_{x,i}\scj U_i$, so that $x\in Z(\text{int}(K_{x,i}))$. This means that the family of open cylinder sets $\{Z(\text{int}(K_{x,i}))\}$ covers $\til{K}$ and so it admits a finite subcover. The corresponding compact sets $K_1,\ldots,K_n$, with corresponding indices $k_1,\ldots,k_n$, are the wanted sets.
\end{proof}

Notice that for compact sets $K_i \scj  Y_{k_i}$, $i=1,\ldots,n$, the union $\bigcup_{i=1}^n K_i$ is a compact subset of $\bigcup_{k\in\nn^*}Y_k$ with the disjoint union topology. Reciprocally, every compact subset of $\bigcup_{k\in\nn^*}Y_k$ is of this form.

\begin{proposition}\label{prop:basis.patch.topology}
	The sets $Z(U)\cap Z(K)^c$ for an open $U\scj Y_k$ and a compact $K\scj \bigcup_{k\in\nn^*}Y_k$ form a basis for the patch topology $\tau_{patch}$.
\end{proposition}

\begin{proof}
	As seen before, the open cylinder sets are closed under intersections. The same is true for the sets $Z(K)^c$ for $K\scj \bigcup_{k\in\nn^*}Y_k$ compact, since, if $L$ is also a compact subset of $\bigcup_{k\in\nn}Y_k$, then so is $K\cup L$, and $Z(K)^c\cap Z(L)^c=(Z(K)\cup Z(L))^c=Z(K\cup L)^c$. This implies that the sets of the form $Z(U)\cap Z(K)^c$ is a basis for some topology $\tau$ on $Y$.
	
	By choosing $K=\emptyset$, we see that $\tau$ is finer then $\tau_{cyl}$. And by choosing $U=Y_1$, and using Lemmas \ref{lemma:compact.saturated.cylinder} and \ref{lemma:compact.saturated.as.union}, we see that $\tau$ is finer than the cocompact topology. By definition, $\tau$ is finer then the patch topology.
	
	On the other hand, the sets $Z(U)\cap Z(K)^c$ are open in the patch topology, and so $\tau_{patch}$ must coincide with $\tau$.
\end{proof}

\begin{remark}\label{remark:open.basic.set}
	For $U\scj Y_k$ open and $K\scj Y_l$, notice that if $l\leq k$, then $Z(U)\cap Z(K)^c=Z(U\cap \pi_{l,k}^{-1}(K^c))$. Hence, we can assume that, in the above proposition, every $K$ is the union of compacts $K_i\scj Y_{l_i}$, $i=1,\ldots,n$ where $l_i>k$ for all $i=1,\ldots,n$.
\end{remark}

One interesting result Paterson and Welch obtained in \cite{MR2146225} is that a sequence in $\prod_{n\in\nn}X_n$ could converge to an element in a finite product $X_1\times\cdots\times X_k$. Since we are now longer in the case of metric spaces, we work with nets instead of sequences. The conditions in the next theorem are a generalization of the sequence converge conditions found in \cite{MR2146225}, \cite{MR2301938} and \cite{MR3679613}.

\begin{theorem}\label{theorem:convergence.nets}
	A net $\{x^\lambda\}_{\lambda\in \Lambda}$ converges to $x$ in $(Y,\tau_{patch})$ if and only if
	\begin{enumerate}[(i)]
		\item\label{item:conv1} For all $1\leq k \leq\ell(x)$ with $k\neq\infty$, there exists $\lambda_0\in \Lambda$ such that for all $\lambda\geq \lambda_0$, $\ell(x^\lambda)\geq k$ and $(x_1^\lambda,\ldots,x_k^\lambda)_{\lambda\geq \lambda_0}$ converges to $(x_1,\ldots,x_{k})$ in $Y_k$.
		\item\label{item:conv2} If $\ell(x)<\infty$, then for any compact $K\scj X_{\ell(x)+1}$, there exists $\lambda_0\in \Lambda$ such that for all $\lambda\geq \lambda_0$, either $\ell(x^\lambda)=\ell(x)$, or $\ell(x^\lambda)>\ell(x)$ and $x_{\ell(x)+1}^\lambda\notin K$.
	\end{enumerate}
\end{theorem}

\begin{proof}
	First suppose that $\{x^\lambda\}_{\lambda\in \Lambda}$ converges to $x$. We first prove (\ref{item:conv1}). Let $1\leq k \leq \ell(x)$ with $k\neq\infty$ be given and consider an open set $U\scj Y_k$ containing $(x_1,\ldots,x_k)$. In this case $x\in Z(U)$ and hence there exists $\lambda_0\in \Lambda$ such that for all $\lambda\geq \lambda_0$, $x^\lambda \in Z(U)$ and therefore $\ell(x^\lambda)\geq k$. By letting $U$ vary, we see that $(x_1^\lambda,\ldots,x_k^\lambda)_{\lambda\geq \lambda_0}$ converges to $(x_1,\ldots,x_{k})$ in $Y_k$.
	
	For (\ref{item:conv2}), let $K\scj X_{\ell(x)+1}$ be compact and $K'$ be a compact neighbourhood of $x$ in $Y_{\ell(x)}$.  Notice that $L=(K'\times K)\cap Y_{\ell(x)+1}$ is compact. By convergence, there exists $\lambda_0\in \Lambda$ such that $x^\lambda\in Z(\mathrm{int}(K'))\cap Z(L)^c$ for all $\lambda\geq \lambda_0$, but this implies that either $\ell(x^\lambda)=\ell(x)$, or $\ell(x^\lambda)>\ell(x)$ and $x_{\ell(x)+1}^\lambda\notin K$.
	
	Now, for the converse suppose that a net $\{x^\lambda\}_{\lambda\in \Lambda}$ and a point $x$ satisfies (\ref{item:conv1}) and (\ref{item:conv2}). Also, let $Z(U)\cap Z(K)^c$ be an open basic set containing $x$. By Remark \ref{remark:open.basic.set}, we can assume that $U\scj Y_k$ for some $k\in\nn^*$ and  $K=\bigcup_{i=1}^n K_i$, where $K_i\scj Y_{k_i}$ with $k_i>k$ for all $i=1,\ldots,n$.
	
	If $\ell(x)=\infty$, by taking $l \geq \max\{k_1,\ldots,k_n\}$ and considering the open set $\pi_{k,l}^{-1}(U)\cup\bigcup_{i=1}^n \pi_{k_i,l}^{-1}(K_i^c)$ in $Y_k$ and using (\ref{item:conv1}), we see that there exists $\lambda_0$ such that $x^\lambda\in Z(U)\cap Z(K)^c$ for all $\lambda\geq \lambda_0$.
	
	In the case that $\ell(x)<\infty$, since $x\in Z(U)\cap Z(K)^c$, we have that $\ell(U)\leq \ell(x)$. We write $K=K'\cup K''$ where $K'=\cup_{i=1}^{n'} K'_i$, with $K'_i\scj Y_{k'_i}$ and $k'_i\leq\ell(x)$, and $K''=\cup_{i=1}^{n''} K''_i$, with $K''_i\scj Y_{k''_i}$ and $k''_i>\ell(x)$. As in the previous case, we find $\lambda_1$ such that for all $\lambda\geq\lambda_1$, $x^{\lambda}\in Z(U)\cap Z(K')^c$. Now, using (\ref{item:conv2}) for each $\pi_{\ell(x)+1,k''_i}(K''_i)$, we find $\lambda_2$ such that $x^{\lambda}\notin Z(K''_i)$ for all $\lambda\geq\lambda_2$ and all $i=1,\ldots,n''$. By taking $\lambda_0$ greater than $\lambda_1$ and $\lambda_2$, we have that for all $\lambda\geq\lambda_0$, $x^{\lambda}\in Z(U)\cap Z(K)^c$.

\end{proof}

We now see that even if we didn't start with Hausdorff spaces, the patch topology is Hausdorff. Also, we consider the property of being locally compact.

\begin{proposition}
	The space $(Y,\tau_{patch})$ is locally compact Hausdorff.
\end{proposition}

\begin{proof}
	Since $\tau_{cyl}$ is a sober topology, $(Y,\tau_{cyl})$ is homeomorphic to $\mathrm{pt}(\tau_{cyl})$. By \cite[Proposition~V.5.12 and Corollary~V.5.13]{MR1975381}, to prove that $(Y,\tau_{patch})$ is locally compact Hausdorff, it is sufficient to show that $(Y,\tau_{cyl})$ is locally compact, which is immediate from Lemma \ref{lemma:compact.saturated.cylinder}.
\end{proof}

%
%

\section{The boundary path space of a graph}\label{section:graph}

In this section we consider the case of directed (discrete) graphs. The idea is to use the definition of the C*-algebra of a graph to impose a relation in the space of all paths, and using the discussion in Section \ref{section:frame}, find a subspace that is exactly the boundary path space. First we recall some of the necessary terminology.

A directed graph is a quadruple $\dgraphuple$, where $\dgraph^0$ and $\dgraph^1$ are sets, whose elements are called vertices and edges respectively, and $d,r:\dgraph^1\to\dgraph^0$ are maps called the domain and range maps\footnote{It is usual to use $s$ instead of $d$ and call it the source map, however we use Katsura's convention \cite{MR2067120}.}. A path is a finite of infinite sequence of edges $\mu=\mu_1\ldots\mu_n(\ldots)$ such that $d(\mu_i)=r(\mu_{i+1})$ for all $i$. The number of edges in path $\mu$ is called its length and denoted by $|\mu|$. Vertices are considered paths of length $0$. The set of all paths of length $n\in\nn\cup\{\infty\}$ is denoted by $\dgraph^n$ and the set of all finite paths (including the vertices) is denoted by $\dgraph^*$. We extend $r$ and $d$ for paths: if $\mu$ is a path of positive length, we define $r(\mu)=r(\mu_1)$ and $d(\mu)=d(\mu_{|\mu|})$; if $\mu$ is a vertex we define $d(\mu)=r(\mu)=\mu$; and if $|\mu|=\infty$ we only define $r(\mu)=r(\mu_1).$

Given $\mu\in\dgraph^*$ and $\nu\in\dgraph^*\cup\dgraph^{\infty}$ such that $r(\nu)=d(\mu)$, then $\mu\nu$ is also a path. We say that a path $\mu$ is a prefix of a path $\xi$ if $\xi=\mu\nu$ for some path $\nu$.

A vertex $v\in\dgraph^0$ is a source if $r^{-1}(v)=\emptyset$, it is an infinite receiver if $r^{-1}(v)$ is infinite, it is a regular vertex if it is neither a source nor an infinite receiver, and it is singular if it is not regular. We denote by $\dgraph_{rg}^0$ and $\dgraph_{sg}^0$ the sets of all regular and singular vertices respectively.

We want to give a topology on the space of all paths $\dgraph^*\cup\dgraph^{\infty}$ using Section \ref{section:topology} results. For that, we consider $X_1=\dgraph^0$ and $X_n=\dgraph^1$ for all $n\geq 2$ as discrete topological spaces. Observe that for each $n\in\nn^*$ there is a bijection between $\dgraph^n$ and the set $Y_{n+1}=\{(x_1,\ldots,x_{n+1})\in X_1\times\cdots\times X_{n+1}\,|\,x_2\ldots x_{n+1}\in\dgraph^n\text{ and }r(x_2)=x_1\}$ (it is basically the graph of the map $r:\dgraph^n\to\dgraph^0$). We also define $Y_1=X_1=\dgraph^0$ and observe that the family $\{Y_n\}_{n\in\nn^*}$ satisfies the conditions in Section \ref{section:topology} and that $Y_{\infty}=\dgraph^{\infty}$. We induce the topology on $\dgraph^*\cup\dgraph^{\infty}$ using the above bijections between $\dgraph^n$ and $Y_{n+1}$.

Since in the discrete topology, compactness is equivalent to finiteness, the patch topology on $\dgraph^*\cup\dgraph^{\infty}$ is the same one described by Webster \cite{MR3119197}. More specifically, from Theorem \ref{theorem:convergence.nets}, we need only to consider the basis $\mathcal{B}$ consisting of sets of the form $Z(\{\mu\})\cap(\bigcup_{e\in F}Z(\{\mu e\}))^c$ for some finite set $F\scj r^{-1}(d(\mu))$. As we will see in Lemma \ref{lemma:cylinder.discrete.case}, in order for $\mathcal{B}$ to be closed under intersections, we also assume that $\emptyset\in\mathcal{B}$.

We now focus our attention to the boundary path space $\partial\dgraph=\{\mu\in\dgraph^*\,|\,d(\mu)\in\dgraph_{sg}^0\}\cup\dgraph^{\infty}$ with the subspace topology. The idea is to impose new relations on a basis for $\dgraph^*\cup\dgraph^{\infty}$ and use Section 2 to find $\partial\dgraph$. Our motivation is the definition of the graph C*-algebra \cite{MR1670363}, which is the universal C*-algebra generated by mutually orthogonal projections $\{p_v\}_{v\in\dgraph^0}$ and partial isometries $\{s_e\}_{e\in\dgraph^1}$ satisfying the relations:
\begin{itemize}
	\item[CK1] $s_e^*s_e=p_{d(e)}$,
	\item[CK2] $s_es_e^*\leq p_{r(e)}$,
	\item[CK3] $p_v=\sum_{e\in r^{-1}(v)}s_es_e^*$, for all $v\in\dgraph_{rg}^0$.
\end{itemize}

Let us interpret the above relations working with homeomorphisms between open subsets of $\dgraph^*\cup\dgraph^{\infty}$. Using the notation of Section \ref{section:topology}, for $v\in\dgraph^0$, define $Z_{v}=Z(\{v\})$ and for $e\in\dgraph^1$, define $Z_{e}=Z(\{e\})$, where we used the identification of $\dgraph^1$ and $Y_2$. For $e\in\dgraph^1$, we also define the map $\sigma_e:Z_{d(e)}\to Z_e$ by $\sigma_e(\mu)=e\mu$. These maps are homeomorphisms with inverse $\sigma_e^{-1}:Z_e\to Z_{d(e)}$ given by $\sigma_e^{-1}(e\mu)=\mu$.

Relation CK1 can be seen as $\sigma_e^{-1}\sigma_e=Id_{Z_{d(e)}}$, and relation CK2 can be seen as $\sigma_e\sigma_e^{-1}\leq Id_{Z_{r(e)}}$, in the sense that $Z_e\scj Z_{r(e)}$. We would like to see CK3 as the equality $Z_v=\cup_{e\in r^{-1}(v)} Z_e$, however this is not true for $v\in Z_v$, but $v\notin Z_e$ for all $e\in r^{-1}(e)$.

We want to use relation CK3 as a motivation to define a coverage on the basis for the patch topology on $\dgraph^*\cup\dgraph^{\infty}$ consisting of sets of the form $Z(\{\mu\})\cap(\bigcup_{e\in F}Z(\{\mu e\}))^c$ for some finite set $F\scj r^{-1}(d(\mu))$.

\begin{lemma}\label{lemma:cylinder.discrete.case}
	Let $\mu,\nu\in\dgraph^*$ be paths, and $F\scj r^{-1}(d(\mu))$, $G\scj r^{-1}(d(\nu))$ be finite sets. Then $Z(\{\nu\})\cap(\bigcup_{f\in G}Z(\{\nu f\}))^c\scj Z(\{\mu\})\cap(\bigcup_{e\in F}Z(\{\mu e\}))^c$ if, and only if, either $\nu=\mu$ and $F\scj G$, or $|\nu|>|\mu|$, $\mu$ is a prefix of $\nu$ and $\nu_{|\mu|+1}\notin F$. Also
	
	\[[Z(\{\mu\})\cap(\bigcup_{e\in F}Z(\{\mu e\}))^c]\cap[Z(\{\nu\})\cap(\bigcup_{f\in G}Z(\{\nu f\}))^c]=\]

	\[=\begin{cases}
	Z(\{\mu\})\cap(\bigcup_{e\in F}Z(\{\mu e\}))^c,\text{ if }|\mu|>|\nu|\text{, }\nu\text{ is a prefix of }\mu\text{ and }\mu_{|\nu|+1}\notin G, \\
	Z(\{\nu\})\cap(\bigcup_{f\in G}Z(\{\mu f\}))^c,\text{ if }|\mu|<|\nu|\text{, }\mu\text{ is a prefix of }\nu\text{ and }\nu_{|\mu|+1}\notin F, \\
	Z(\{\mu\})\cap(\bigcup_{e\in F\cup G}Z(\{\mu e\}))^c,\text{ if }\mu=\nu, \\
	\emptyset,\text{ otherwise.}
	\end{cases}\]
	
\end{lemma}

\begin{proof}
	 Noticing that the elements of the set $Z(\{\mu\})\cap(\bigcup_{e\in F}Z(\{\mu e\}))^c$ are paths that have $\mu$ as a prefix, but not $\mu e$ for all $e\in F$, the proof is straightforward.
\end{proof}

To define a coverage of $\mathcal{B}$, let $a=Z(\{\mu\})\cap(\bigcup_{e\in F}Z(\{\mu e\}))^c\in\mathcal{B}$ be given. If $d(\mu)\in\dgraph^0_{sg}$ define $C'(a)=\{\{\emptyset, a\}\}$. If $d(\mu)\in\dgraph^0_{rg}$ and $F\subseteq r^{-1}(d(\mu))$, define $C'(a)=\{\{\emptyset, a\},\{\emptyset\}\cup\{Z(\{\mu e\})\}_{e\in r^{-1}(d(\mu))\cap F^c}\}$. Also, $C'(\emptyset)=\{\{\emptyset\}\}$.

\begin{proposition}
	The above families give a coverage $C'$ of $\mathcal{B}$.
\end{proposition}

\begin{proof}
	Let $a,b\in\mathcal{B}$ be such that $b\scj a$. We need to check the meet stability property. If either $a$ or $b$ is the empty set, the result is trivial. Suppose then
	that $a=Z(\{\mu\})\cap(\bigcup_{e\in F}Z(\{\mu e\}))^c$ and $b=Z(\{\nu\})\cap(\bigcup_{f\in G}Z(\{\nu f\}))^c$.
	
	If $d(\mu)\in\dgraph^0_{sg}$, we only need to check that $\{\emptyset\cap b,a\cap b\}\in C'(b)$, which is true since $a\cap b=b$. If $d(\mu)\in\dgraph^0_{rg}$, it is still the case that $\{\emptyset\cap b,a\cap b\}\in C'(b)$.
	
	Finally, suppose that $F\subseteq r^{-1}(d(\mu))$ and consider the following covering $\{\emptyset\}\cup\{Z(\{\mu e\})\}_{e\in r^{-1}(d(\mu))\cap F^c}$. Using Lemma \ref{lemma:cylinder.discrete.case}, we have two cases for $b$. The first case is when $\mu=\nu$ and $F\scj G$. In this case, by the same lemma, $b\cap Z(\{\mu e\})=\emptyset$ if $e\in G$ and $b\cap Z(\{\mu e\})=Z(\{\mu e\})$ if $e\in r^{-1}(d(\mu))\cap G^c$. Thus $\{b\cap \emptyset\}\cup\{b\cap Z(\{\mu e\})\}_{e\in r^{-1}(d(\mu))\cap F^c}=\{\emptyset\}\cup\{Z(\{\mu e\})\}_{e\in r^{-1}(d(\mu))\cap G^c}\in C'(b)$. The second case is when $|\nu|>|\mu|$, $\mu$ is a prefix of $\nu$ and $\nu_{|\mu|+1}\notin F$. In this case $b\cap Z(\{\mu e\})=\emptyset$ if $e\neq \nu_{|\mu|+1}$ and $b\cap Z(\{\mu e\})=b$ if $e=\nu_{|\mu|+1}$. Hence $\{b\cap \emptyset\}\cup\{b\cap Z(\{\mu e\})\}_{e\in r^{-1}(d(\mu))\cap F^c}=\{\emptyset,b\}\in C'(b)$.
\end{proof}

\begin{theorem}\label{theorem:bondaury.discrete.graph}
	Let $\overline{C}$ and $Y$ be as in Lemma \ref{lemma:subspace.generated.by.relations} with $C'$ as above. Then $Y=\partial\dgraph$.
\end{theorem}

\begin{proof}
	We use Remark \ref{remark:subspace}.
	
	We begin showing that $Y\scj \partial\dgraph$. For that, we show that if $\xi\in\dgraph^*\cup\dgraph^{\infty}$ is such that $\xi\in\dgraph^*$ and $d(\xi)\in\dgraph^0_{rg}$, then $\xi\notin Y$. Consider $a=Z(\{\xi\})$ and the covering $S=\{\emptyset\}\cup\{Z(\{\xi e\})\}_{e\in r^{-1}(d(\xi))}$, then $\xi\notin b$ for all $b\in S$, however $\xi\in a$.
	
	Now we prove that $\partial\dgraph\scj Y$. Consider $\xi\in\dgraph^*$ with $d(\xi)\in\dgraph^0_{sg}$, and $a=Z(\{\mu\})\cap(\bigcup_{e\in F}Z(\{\mu e\}))^c\in\mathcal{B}$. If $\xi\in a$, then $\mu$ is a prefix of $\xi$. Supposing that $d(\mu)\in\dgraph^0_{sg}$, which is the case if $\mu=\xi$, then the only covering $S$ of $a$ is $S=\{\emptyset,a\}$ and clearly $\xi\in b$ for some $b\in S$. Supposing that $d(\mu)\in\dgraph^0_{rg}$, then $|\mu|<|\xi|$ and $\xi_{|\mu|+1}\neq e$ for all $e\in F$. For the covering $S=\{\emptyset\}\cup\{Z(\{\mu e\})\}_{e\in r^{-1}(d(\mu))\cap F^c}$, $\xi_{|\mu|+1}\in r^{-1}(d(\mu))\cap F^c$ and $\xi \in Z(\{\mu\xi_{|\mu|+1}\})$. Again, we conclude that $\xi\in b$ for some $b\in S$. This implies that $\xi\in Y$.
	
	If $\xi\in\dgraph^{\infty}$. As in the second part of the previous paragraph, we see that if $\xi\in a$, then for all covering $S\in C'(a)$, $\xi \in b$ for some $b\in S$. Again, $\xi\in Y$.	
\end{proof}

%
%

\section{The boundary path space of a topological graph}\label{section:topolgical.graph}

As defined by Katsura \cite{MR2067120}, a topological graph is a quadruple $\dgraphuple$ such that $\dgraph^0$, $\dgraph^1$ are locally compact Hausdorff spaces, $r:\dgraph^1\to \dgraph^0$ is a continuous function and $d:\dgraph^1\to\dgraph^0$ is a local homeomorphism. We use the same terminology of paths as in the previous section, but we change the definitions of sources, infinite receivers, regular and singular vertices in order to take into account the topologies on $\dgraph^0$ and $\dgraph^1$.

The following sets were defined by Katsura in \cite{MR2067120}: $\dgraph^0_{sce}=\dgraph^0\setminus\overline{r(\dgraph^1)}$,
\[\dgraph_{fin}^0=\{v\in\dgraph^0\,|\,\exists V\text{neighbourhood of }v\text{ such that }r^{-1}(V)\scj\dgraph^1\text{ is compact}\},\]
$\dgraph^0_{rg}=\dgraph^0_{fin}\setminus\overline{\dgraph^0_{sce}}$, which are open subsets of $\dgraph^0$, $\dgraph_{inf}^0=\dgraph^0\setminus\dgraph_{fin}^0$ and $\dgraph_{sg}^0=\dgraph^0\setminus\dgraph_{rg}^0=\dgraph_{inf}^0\cup \overline{\dgraph^0_{sce}}$, which are closed subsets of $\dgraph^0$. With this new definition of $\dgraph_{sg}^0$, the boundary path space of $\dgraph$ is again $\partial\dgraph=\{\mu\in\dgraph^*\,|\,d(\mu)\in\dgraph_{sg}^0\}\cup\dgraph^{\infty}$.

We will need the following result by Katsura.

\begin{proposition}[\cite{MR2067120}, Proposition 2.8]\label{proposition:regular.vertex.top.graph}
	For $v\in\dgraph^0$, we have that $v\in\dgraph_{rg}^0$ if and only if there exists a neighbourhood $V$ of $v$ such that $r^{-1}(V)\scj\dgraph^1$ is compact and $r(r^{-1}(V))=V$.
\end{proposition}

In this section, the maps $r,d$ when considered defined on the set of paths of length $n\in\nn$ are also denoted by $r_n$ and $d_n$ respectively. For each $n\in\nn$, $d_n$ is also a local homeomorphism and, in particular, an open map.

We want to consider the patch topology of Section \ref{section:topology} on $\dgraph^*\cup\dgraph^{\infty}$. As in the previous section, for each $n\in\nn^*$, we see $\dgraph^n$ as a subset of $\dgraph^0\times(\dgraph^1)^{\times n}$, which is closed because $r$ and $d$ are continuous. For $n\in\nn$, if $U\scj\dgraph^n$ is an open set, we write $|U|=n$ (or $|U|=\ell(U)-1$ with the notation of Section \ref{section:topology}). Also, for $n\in\nn^*$ and $1\leq k\leq n$, we denote by $U_k$ the projection of $U$ on the $k$-th coordinate, and for $1\leq k\leq l\leq n$, $U_{k,l}$ is the projection of $U$ from the $k$-th to the $l$-th coordinates, so that $U_{k,l}$ consists of paths of length $l-k+1$.
 
The goal of this section is to define a covering on the basis $\mathcal{B}$ given by Proposition \ref{prop:basis.patch.topology} in such way that the subspace of Lemma \ref{lemma:subspace.generated.by.relations} is again $\partial\dgraph$. We assume that $\emptyset\in\mathcal{B}$. Let us first show that $\mathcal{B}$ is a semilattice.

\begin{lemma}\label{lemma:intersection.topological}
	For $U\subseteq\dgraph^n$ and $V\subseteq\dgraph^m$ for some $m,n\in\nn$, we have that
	\[Z(U)\cap Z(V)=
	\begin{cases}
	Z((U\times r^{-1}_{m-n}(d_n(U)))\cap V) & \text{if } n<m, \\
	Z(U \cap (V\times r^{-1}_{n-m}(d_m(V)))) & \text{if } m<n, \\
	Z(U \cap V) & \text{if } n=m.
	\end{cases}\]
\end{lemma}

\begin{proof}
	If $n=m$, then both sides are the set of all paths for which the prefix is in $U\cap V$.
	
	Suppose that $n<m$ and fix a path $\mu\nu\xi$ with $|\mu|=n$, $|\nu|=m-n$ and $\xi\in \dgraph^*\cup\dgraph^{\infty}$. Then, $\mu\nu\xi\in Z(U)\cap Z(V)$, if and only if, $\mu \in U$, $\mu\nu \in V$ and $r(\nu)=d(\mu)\in d(U)$, if and only if, $\mu\nu\xi\in Z((U\times r^{-1}_{m-n}(d_n(U)))\cap V)$.
	
	The case $m<n$ is analogous. 
\end{proof}

Since for every $n\in\nn$, $r_n$ is continuous and $d_n$ is an open map, observing that for $K,L\scj\dgraph^*$ compact in the disjoint union topology, $Z(K)^c\cap Z(L)^c=Z(K\cup L)^c$, the previous lemma implies that $\mathcal{B}$ is a semilattice. Also, observe that $Z(V)\cap Z(L)^c\subseteq Z(U)\cap Z(K)^c$ if and only if $V\subseteq (U\times r^{-1}_{|V|-|U|}(d(U)))\cap\dgraph^{|V|}$ and $L\supseteq K$.

We will define two kind of coverings which we will call continuation coverings and topological coverings. First, motivated by last section, for $A\subseteq r(\dgraph^1)$ open in $\dgraph^0$, we define a continuation covering of $A$ as a finite family of pairwise disjoint open subsets of $\dgraph^1$, $\{\emptyset,B_1,\ldots,B_k\}$, such that for each $i=1,\ldots,k$, $d|_{B_i}$ is as homeomorphism onto its image, $B_i$ is relatively compact, and
\[r^{-1}(A)=\bigcup_{i=1}^k B_i.\]

Let $U\scj\dgraph^n$ be an open set such that $d(U)\subseteq r(\dgraph^1)$ and $K\scj\dgraph^*$ compact, a continuation covering for $Z(U)\cap Z(K)^c$ is a family $\{\emptyset,Z((U\times B_1)\cap\dgraph^{n+1})\cap Z(K)^c,\ldots,Z((U\times B_k)\cap\dgraph^{n+1})\cap Z(K)^c\}$ such that $\{\emptyset,B_1,\ldots,B_k\}$ is a continuation covering of $d(U)$. Let $C'(Z(U)\cap Z(K)^c)$ be the set of all these families, if they exist, together with the topological coverings $\{\emptyset, Z(U_\lambda)\cap Z(K)^c\}_{\lambda\in\Lambda}$ where $\{U_\lambda\}_{\lambda\in\Lambda}$ is a family of open subsets of $\dgraph^n$ such that $U=\bigcup_{\lambda\in\Lambda}U_\lambda$.

\begin{proposition}
	The above families give a coverage $C'$ on $\mathcal{B}$.
\end{proposition}

\begin{proof}
	We have to check the meet-stability property. Let $Z(V)\cap Z(L)^c\subseteq Z(U)\cap Z(K)^c$ and $S$ be a covering of $Z(U)\cap Z(K)^c$. If the covering is a topological covering, it is immediate from Lemma \ref{lemma:intersection.topological}. Suppose then that $S=\{\emptyset,Z((U\times B_1)\cap\dgraph^{n+1})\cap Z(K)^c,\ldots,Z((U\times B_k)\cap\dgraph^{n+1})\cap Z(K)^c\}$ is a continuation covering.
	
	In the case that $|V|=|U|=n$, we have that $V\subseteq U$, so that $d(V)\scj d(U)\scj r(\dgraph^1)$, $K\scj L$ and 
	\begin{align*}(Z(V)\cap Z(L)^c)&\cap(Z((U\times B_i)\cap\dgraph^{n+1})\cap Z(K)^c) =\\
	&=Z((V\times r^{-1}(d(V)))\cap (U\times B_i)\cap\dgraph^{n+1})\cap Z(K\cup L)^c \\
	&=Z((V\cap U)\times (r^{-1}(d(V))\cap B_i)\cap\dgraph^{n+1})\cap Z(L)^c \\
	&=Z((V \times (r^{-1}(d(V))\cap B_i))\cap\dgraph^{n+1})\cap Z(L)^c.
	\end{align*}
	Since $r^{-1}(d(V)) \scj r^{-1}(d(U))$, then $r^{-1}(d(V))=\cup_{i=1}^k r^{-1}(d(V))\cap B_i$. Clearly the sets $r^{-1}(d(V))\cap B_i$ for $i=1,\ldots,k$ are pairwise disjoint, each one of them is relatively compact, and each $d|_{r^{-1}(d(V))\cap B_i}$ is a homeomorphism onto its image.
	
	Now suppose that $m=|V|>|U|=n$ (we assume below that $m>n+1$, the case $m=n+1$ being analogous). In this case, $V_{1,n}\scj U$, $V_{n+1}\scj r^{-1}(d(U))$, $K\scj L$, and 
	\begin{align*}
	(Z&(V)\cap Z(L)^c)\cap(Z((U\times B_i)\cap\dgraph^{n+1})\cap Z(K)^c)=\\
	 &=Z((V_{1,n}\times V_{n+1} \times V_{n+2,m}\cap (U\times B_i \times r_{m-n-1}^{-1}(d(B_i)))\cap\dgraph^{m})\cap Z(K\cup L)^c \\
	 &=Z(V_{1,n}\times (V_{n+1}\cap B_i) \times (V_{n+2,m} \cap r_{m-n-1}^{-1}(d(B_i)))\cap\dgraph^m)\cap Z(L)^c.
	\end{align*}
	
	Due to path lengths, the family of these sets cannot form a continuation covering of $Z(V)\cap Z(L)^c$. Therefore, we have to prove that the union of the sets $V_{1,n}\times (V_{n+1}\cap B_i) \times (V_{n+2,m} \cap r_{m-n-1}^{-1}(d(B_i)))\cap\dgraph^m$ for $i=1,\ldots,k$ is equal to $V$. For the middle part,
	\[V_{n+1}\scj r^{-1}(d(U))=\bigcup_{i=1}^k B_i\]
	and hence
	\[\bigcup_{i=1}^k V_{n+1}\cap B_i =  V_{n+1}\cap \bigcup_{i=1}^k B_i = V_{n+1}.\]
	For the tail,
	\[V_{n+2,m}\scj r_{m-n-1}^{-1}(d(V_{n+1}))\scj r_{m-n-1}^{-1}\left(d\left(\bigcup_{i=1}^k B_i\right)\right)\]
	so that
	\[V_{n+2,m}=V_{n+2,m}\cap r_{m-n-1}^{-1}\left(d\left(\bigcup_{i=1}^k B_i\right)\right) = \bigcup_{i=1}^k V_{n+2,m} \cap r_{m-n-1}^{-1}(d(B_i)).\]
	
	It follows that
	\[\bigcup_{i=1}^k V_{1,n}\times (V_{n+1}\cap B_i) \times (V_{n+2,m} \cap r_{m-n-1}^{-1}(d(B_i)))\cap\dgraph^m=V\cap\dgraph^m=V.\]
\end{proof}

\begin{theorem}
	The subspace $Y$ of $\dgraph^*\cup\dgraph^{\infty}$ defined by the above coverage as in Lemma \ref{lemma:subspace.generated.by.relations} is the boundary path space of the topological graph.
\end{theorem}

\begin{proof}
	We mimic the proof of Theorem \ref{theorem:bondaury.discrete.graph} using Remark \ref{remark:subspace}. The coverage $\overline{C}$ is as in Lemma \ref{lemma:subspace.generated.by.relations}.
	
	To prove that $Y\scj \partial\dgraph$, we prove that if $\mu\in\dgraph^*$ is such that $d(\mu)\in\dgraph_{rg}^0$, then $\mu\notin Y$. For that, we need to find a open set $a$ of $\dgraph^*\cup\dgraph^{\infty}$  and a covering $S\in \overline{C}(a)$ such that $\mu\notin b$ for all $b\in S$, but $\mu\in a$.
	
	Let $V\scj\dgraph^0$ be as in Proposition \ref{proposition:regular.vertex.top.graph}. For $n=|\mu|$, since $d=d_n:\dgraph^n\to\dgraph^0$ is a local homeomorphism, there exists $U\scj\dgraph^n$ open such that $\mu\in U$, $d|_U$ is a homeomorphism onto its image and $d(U)\scj \mathrm{Int}(V)\scj V=r(r^{-1}(V))\scj r(\dgraph^1)$. Since $r^{-1}(V)\scj\dgraph^1$ is compact, $\dgraph^1$ is locally compact and $d:\dgraph^1\to\dgraph^0$ is a local homeomorphism, there exists a finite open cover $\{A_1,\ldots,A_k\}$ of $r^{-1}(V)$ such that $d|_{A_i}$ is a homeomorphism onto its image and $A_i$ is relatively compact for all $i=1,\ldots,k$. For each $i$, define $B_i=r^{-1}(d(U))\cap A_i$ and notice that $d|_{B_i}$ is still a homeomorphism onto its image. Also $B_i\scj r^{-1}(V)$ so that it is relatively compact and 
	\[\bigcup_{i=1}^k B_i=\bigcup_{i=1}^k r^{-1}(d(U))\cap A_i= r^{-1}(d(U))\cap\bigcup_{i=1}^k A_i=r^{-1}(d(U)).\]
	By taking $a=Z(U)$ and $S=\{\emptyset,Z((U\times B_1)\cap\dgraph^{n+1}),\ldots,Z((U\times B_k)\cap\dgraph^{n+1})\}$, we have that $\mu\in a$, but $\mu\notin b$ for all $b\in S$.
	
	Now, let us prove that $\partial\dgraph\scj Y$. First observe that if $\mu\in Z(U)\cap Z(K)^c$ for some element of $\mathcal{B}$, then for any topological covering $S$ of $Z(U)\cap Z(K)^c$, there exists $b\in S$ such that $\mu\in b$. Indeed, if $S=\{\emptyset,Z(U_\lambda)\cap Z(K)^c\}_{\lambda\in \Lambda}$ with $\bigcup_{\lambda\in\Lambda}U_\lambda=U$, then, since $\mu_{1,|U|}\in U$, we have that $\mu_{1,|U|}\in U_\lambda$ for some $\lambda\in \Lambda$, and therefore $\mu\in Z(U_\lambda)\cap Z(K)^c$. In the case that $|\mu|>|U|$	if $S=\{\emptyset,Z((U\times B_1)\cap\dgraph^{|U|+1})\cap Z(K)^c,\ldots,Z((U\times B_k)\cap\dgraph^{|U|+1})\cap Z(K)^c\}$ is a continuation covering, since $\mu_{|U|+1}\in r^{-1}(d(U))$, there exists $i=1,\ldots,k$ such that $\mu_{|U|+1}\in B_i$, in which case $\mu\in Z((U\times B_i)\cap\dgraph^{|U|+1})\cap Z(K)^c$.
	
	From the above discussion, we see that if $\mu\in\dgraph^{\infty}$, then $\mu\in Y$, since $|\mu|>n$ for all $n\in\nn$. And for $\mu\in\dgraph^*$ with $d(\mu)\in\dgraph_{sg}^0$, it is sufficient to prove that if $\mu\in Z(U)\cap Z(K)^c$ for some element of $\mathcal{B}$ with $|\mu|=|U|$, then $Z(U)\cap Z(K)^c$ does not admit a continuation covering. In the case that $d(\mu)\in\overline{\dgraph_{sce}^0}$, any neighbourhood of $d(\mu)$ contains at least a vertex $v\notin r(\dgraph^1)$. In particular $d(U)\nsubseteq r(\dgraph^1)$ so that it does not admit a continuation covering, which implies the same for $Z(U)\cap Z(K)^c$. Finally, in the case that $d(\mu)\in\dgraph_{inf}^0$, since $\dgraph^0$ is locally compact, there exists $V\scj d(U)$ compact neighbourhood of $d(\mu)$, and by the definition of $\dgraph_{inf}^0$, $r^{-1}(V)$ is not compact. If $r^{-1}(d(U))=\bigcup_{i=1}^k B_i$ for some open sets $B_i\scj \dgraph^1$, $i=1,\ldots,k$, then $r^{-1}(V)\scj\overline{\bigcup_{i=1}^k B_i}=\bigcup_{i=1}^k \overline{B_i}$. This implies that $B_i$ is not relatively compact for some $i=1,\ldots,k$, and hence $Z(U)\cap Z(K)^c$ does not admit a continuation covering.
\end{proof}

\bibliographystyle{abbrv}
\bibliography{sgtg_ref}
	
\end{document}